\documentclass[12pt]{amsart}
\usepackage{amssymb,amsmath,amsthm}
\usepackage{amscd}
\input amssym.def
\input amssym
\newtheorem{theorem}{Theorem}[section]

\newtheorem{remark}{Remark}[section]
\newtheorem{defi}{Definition}[section]
\newtheorem{prop}{Proposition}[section]

\newcommand{\be}{\begin{equation}}
\newcommand{\ee}{\end{equation}}

\renewcommand{\theequation}{\thesection.\arabic{equation}}
\renewcommand{\thetheorem}{\thesection.\arabic{theorem}}
\renewcommand{\theequation}{\thesection.\arabic{equation}}
\begin{document}

\title[] {A classical vertex algebra constructed with the use
  of some logarithmic formal calculus}

\author{Thomas J. Robinson}


\begin{abstract}
Using some new logarithmic formal calculus, we construct a well known
vertex algebra, obtaining the Jacobi identity directly, in an
essentially self-contained treatment.
\end{abstract}

\maketitle

\renewcommand{\theequation}{\thesection.\arabic{equation}}
\renewcommand{\thetheorem}{\thesection.\arabic{theorem}}
\setcounter{equation}{0} \setcounter{theorem}{0}
\setcounter{section}{0}
\section{Introduction}
Using some new logarithmic formal calculus, we construct a well known
vertex algebra associated with $\widehat{\goth{sl}(2)}$, obtaining the
Jacobi identity directly, in an essentially self-contained treatment.
This treatment is largely a (very) special case following the
development of Chapter 8 and the relevant preliminaries in \cite{FLM}
(see also the elementary parts of \cite{LL}, \cite{HLZ} and \cite{M1}).
We introduce certain novelties, mostly having to do with some
logarithmic calculus which grew out of considering a certain heuristic
Remark 4.2.1 in \cite{FLM}.  The main idea is to systematically
consider expressions like $e^{\log x}$, a formal power series in the
formal variable $\log x$.  Of course, $e^{\log x}$ is heuristically
equivalent to $x$ itself, and we suitably and carefully make this and
similar identifications rigorous.  The reader need not know anything
about $\widehat{\goth{sl}(2)}$, and indeed it is not even mentioned
again in this paper.  We only note it for those who already know the
answer.

We wish to direct the reader to an independent use of similar formal
logarithmic objects as those considered in this paper.  In \cite{M2}
(see especially Section 8), the author rigorously uses certain
exponentiated logarithmic terms similar to those considered here in
order to construct certain operators which they call ``mock
logarithmic intertwining operators.''

We wish also to acknowledge helpful discussions with and valuable
comments from Francesco Fiordalisi, Jim Lepowsky, Antun Milas and
Robert Wilson.

\section{Some preliminary formal calculus}
\setcounter{equation}{0} We work in the standard formal calculus
setting of \cite{FLM} and \cite{LL} and we start by briefly recalling
some elementary facts treated in those works.  We shall write $x,y,z$
as well as $x_{0},x_{1},x_{2},x_{3},\dots,y_{0},y_{1}, \dots,
\xi_{0},\xi_{1}, \dots$ etc. for commuting formal variables.  In this
paper, formal variables will always commute, and we will not use
complex variables.  All vector spaces will be over $\mathbb{C}$,
although one may easily generalize many results to the case of a field
of characteristic $0$.  Let $V$ be a vector space.  We use the
following:
\begin{align*}
V[[x,x^{-1}]]=\biggl\{ \sum_{n \in \mathbb{Z}}v_{n}x^{n}|v_{n} \in V  \biggr\}
\end{align*}
(formal Laurent series),
and some of its subspaces:
\begin{align*}
V((x))=\biggl\{ \sum_{n \in \mathbb{Z}}v_{n}x^{n}|v_{n} \in V,\, v_{n}=0
\text{ for sufficiently negative } n \biggr\}
\end{align*}
(truncated formal Laurent series),
\begin{align*}
V[[x]]=\biggl\{ \sum_{n  \geq 0}v_{n}x^{n}|v_{n} \in V \biggr\}
\end{align*}
(formal power series),
\begin{align*}
V[x,x^{-1}]=\biggl\{ \sum_{n \in \mathbb{Z}}v_{n}x^{n}|v_{n} \in V,\,
v_{n}=0 \text{ for all but finitely many } n \biggr\}
\end{align*}
(formal Laurent polynomials), and
\begin{align*}
V[x]=\biggl\{ \sum_{n \geq 0}v_{n}x^{n}|v_{n} \in V,\, v_{n}=0 \text{ for all
but finitely many } n \biggr\}
\end{align*}
(formal polynomials).  Often our vector space $V$ will be a vector
space of endomorphisms, $\text{\rm End}\,V$.  Even when $V$ is
replaced by $\text{\rm End}\,V$ some of these spaces are not algebras,
and we must define multiplication only up to a natural restrictive
condition, a {\it summability condition}.  Let $f_{i}(x)=\sum_{r \in
\mathbb{C}}a_{i}(r)x^{r} \in (\text{\rm End}\, V)\{x \}$ for $1 \leq i
\leq m$.  Then the product
\begin{align*}
f_{1}(x)f_{2}(x)\cdots f_{k}(x)
\end{align*}
exists if for every $m \in \mathbb{C}$ and $v \in V$
\begin{align*}
\sum_{n_{1}+\cdots+n_{k}=m}a_{1}(n_{1})\cdots a_{k}(n_{k})v
\end{align*}
is a finite sum.  We will use routine extensions of this principle of
summability without further comment.  One must be careful with
existence issues when dealing with associativity properties of such
products.  For instance, if $F(x), G(x)$ and $H(x) \in \text{\rm
  End}\,V[[x,x^{-1}]]$ and if the three products $F(x)G(x)$,
$G(x)H(x)$ and $F(x)G(x)H(x)$ all exist, then it is easy to verify
that
\begin{align*}
(F(x)G(x))H(x)=F(x)(G(x)H(x)),
\end{align*}
but unless some such condition is checked then ``associativity'' may
fail.  We shall use this type of associativity and routine extensions
of the principle without further comment.
\begin{remark} \rm 
  Throughout this paper, we often extend our spaces
  to include more than one variable.  We state certain properties
  which have natural extensions in such multivariable settings, which
  we will also use without further comment.
\end{remark}
We shall frequently use the notation $e^{w}$ to refer to the
formal exponential expansion, where $w$ is any formal object for which
such expansion makes sense (meaning that the coefficients are finitely
computable).  For instance, we have the linear operator
$e^{y\frac{d}{dx}}:\mathbb{C}[[x,x^{-1}]] \rightarrow
\mathbb{C}[[x,x^{-1}]][[y]]$:
\begin{align*}
e^{y\frac{d}{dx}}=\sum_{n \geq
0}\frac{y^{n}}{n!}\left(\frac{d}{dx}\right)^{n}.
\end{align*}
Simliarly we let $\log(1+w)$ refer to the formal logarithmic series
\begin{align*}
\log(1+w)=\sum_{i \geq 1}\frac{(-1)^{i+1}}{i}w^{i},
\end{align*}
whenever the expansion makes sense.  We note that we shall also be
using objects with names like $\log x$ which is a symbol for a new
atomic object, not a shorthand for a series, but rather a single new
formal variable, and it should not be confused with $\log (1+x)$ which
is a series in the entirely different formal variable, $x$.

\begin{prop} (The ``automorphism property'') Let $A$ be an algebra
  over $\mathbb{C}$.  Let $D$ be a formal derivation on $A$.  That is,
  $D$ is a linear map from $A$ to itself which satisfies the product
  rule:
\begin{align*}
D(ab)=(Da)b+a(Db), \text{ for all } a\text{ and }b \text{ in }A.
\end{align*}
Then 
\begin{align*}
e^{yD}(ab)=\left(e^{yD}a\right)\left(e^{yD}b\right),
\end{align*}
\end{prop}
\begin{proof}
Notice that

$$
D^{n}ab=\sum_{l+k=n}^{r}n!\frac{D^{k}a}{k!}\frac{D^{l}b}{l!}.
$$
Then divide both sides by $n!$ and sum over $y$ and the result follows.
\end{proof}
Essentially for the same reason as the above proof (combinatorially
speaking) we have for $w_{1}$ and $w_{2}$ commuting objects, the
formal rule
\begin{align}
\label{eq:expcomm1}
e^{w_{1}+w_{2}}=e^{w_{1}}e^{w_{2}}.
\end{align}
More generally, we shall find the following formal rule very useful.
If $[x,y]$ commutes with $x$ and $y$ then
\begin{align}
\label{eq:ecomm}
e^{x}e^{y}=e^{y}e^{x}e^{[x,y]}.
\end{align}
We essentially follow the exposition for formula (3.4.7) in
\cite{FLM}.
\begin{align*}
xe^{y}
=x\sum_{n \geq 0}\frac{y^{n}}{n!}
=\sum_{n \geq 0}\frac{y^{n}}{n!}x + \sum_{n \geq 1}\frac{ny^{n-1}}{n!}[x,y]
=e^{y}x+e^{y}[x,y]
=e^{y}(x+[x,y]),
\end{align*}
and iterating gives
\begin{align*}
x^{k}e^{y}=e^{y}(x+[x,y])^{k},
\end{align*}
and dividing by $k!$ and summing over $k \geq 0$ gives
\begin{align*}
e^{x}e^{y}=e^{y}e^{x+[x,y]},
\end{align*}
which gives the result because $[x,y]$ commutes with $x$.

We define 
\begin{align}
\label{eq:bin}
\binom{n}{k}=\frac{n(n-1)(n-2) \cdots (n-k+1)}{k!}
\end{align}
for all $n \in \mathbb{C}$ and $k \in \mathbb{Z}$.
\begin{defi}
\label{def:trans}
\begin{align*}
(x+y)^{r}=e^{y\frac{d}{dx}}x^{r} \qquad \text{for} \quad r \in \mathbb{C}.
\end{align*}
\end{defi}
This definition is equivalent to the usual {\it binomial expansion
  convention}, the convention being that the second listed variable is
expanded in nonegative powers.  The reader may easily check that this
definition coincides with the usual definition of $(x+y)^{n}$ when $n$
is a nonegative integer.  Other expansions, using (\ref{eq:bin}), also
``look'' as expected.  But we must be careful.  The following
nontrivial fact is true.
\begin{prop}
\label{prop:assocarith}
For all $n \in \mathbb{Z}$,
\begin{align*}
(x+(y+z))^{n}=((x+y)+z)^{n}.
\end{align*}
\end{prop}
\begin{proof}
  Recalling (\ref{eq:expcomm1}), we have
\begin{align*}
(x+(y+z))^{n}
=e^{(y+z)\frac{\partial}{\partial x}}x^{n}
=e^{y\frac{\partial}{\partial x}}
\left(e^{z\frac{\partial}{\partial x}}x^{n}\right)
=e^{y\frac{\partial}{\partial x}}(x+z)^{n}
=((x+y)+z)^{n}.
\end{align*}
\end{proof}

Given formal commuting variables $x$ $y$ etc., we take the point of
view of certain elementary material in \cite{M1} and \cite{HLZ} and
let $\log x$, $\log y$ etc. be formal variables commuting with $x$,
$y$ and each other etc.  We further require that $\frac{d}{dx}\log
x=x^{-1}$ and extend the operator to act as the unique possible
derivation on $\mathbb{C}[x,\log x]$.  We further ``formally linearly
complete'' $\frac{d}{dx}$ to act on $\mathbb{C}[[x,x^{-1},\log x,
(\log x)^{-1}]]$ in the obvious way.  We shall use other similar
routine definitions without further comment.
\begin{defi}
\begin{align*}
(\log (x+y))^{r}=e^{y\frac{d}{dx}} (\log x)^{r} \qquad r \in \mathbb{C}.
\end{align*}
\end{defi}
It is easy to verify that
\begin{align*}
\log(x+y)=\log x+\log\left(1+\frac{y}{x}\right).
\end{align*}
The following is a special case of the first part of Theorem 3.6 in
\cite{HLZ}.
\begin{prop} 
\label{prop:logtaylor}
(The logarithmic formal Taylor
  theorem)\\ For $p(x) \in \mathbb{V}[[x,x^{-1},\log x, (\log x)^{-1}]]$,
\begin{align*}
e^{y\frac{d}{dx}}p(x)=p(x+y).
\end{align*} 
\end{prop}
\begin{proof}
  A summand of $p(x)$ is doubly weighted by the degrees of the formal
  variables.  The summands of $p(x+y)$ are doubly weighted by the
  total degree in $x$ and $y$ as well as separately in the degree in
  $\log x$.  The operator preserves the total degree of the $x$ and
  $y$ variables.  Thus we only need to consider $p(x)$ of the form
  $x^{k}q(x)$ where $k \in \mathbb{Z}$ and $q(x) \in \mathbb{V}[[\log
  x, (\log x)^{-1}]]$.  Let $q_{n}(x)$ be the Laurent polynomial which
  is equal to $q(x)$ but truncated in powers whose absolute value is
  less than $n$.  If we focus on a fixed power of $y$ and separately
  of $\log x$ it is clear that we only need to consider $q_{n}(x)$ in
  place of $q(x)$.  The result follows by linearity after considering
  the trivial cases $p(x)=x^{r}$ and $p(x)=(\log x )^{r}$ for $r \in
  \mathbb{C}$ and applying the automorphism property.
  \end{proof}

\begin{prop}
The vectors $e^{mx}$ for $m \in \mathbb{C}$ are linearly independent.
\end{prop}
\begin{proof}
Otherwise there is a linear combination $\sum_{m \in
\mathbb{C}}a_{m}e^{mx}=0$, where $a_{m} \in \mathbb{C}$ and only
finitely many coefficients are nonzero.  Considering the higher
derivatives and focusing on the constant term yields a Vandermonde
matrix which is nonsingular which gives that the coefficients must all
be zero.
\end{proof}

We have therefore that the space spanned by the vectors $e^{m\xi_{0}}$
$m \in \mathbb{Z}$ is the space of Laurent polynomials
$\mathbb{C}[e^{\xi_{0}},e^{-\xi_{0}}]$.  We tensor this space with the
space of polynomials $\mathbb{C}[\xi_{1},\xi_{2}, \dots]$  to get the space
\begin{align}
\label{Xi}
\Xi=\mathbb{C}[e^{\xi_{0}},e^{-\xi_{0}},\xi_{1},\xi_{2}, \dots].
\end{align}
Our vertex algebra will have $\Xi$ as its underlying space.  Of
course, $\Xi$ is also an algebra under the obvious rules.

\section{The delta function}
\setcounter{equation}{0}
We define the formal delta function by
\begin{align*}
\delta (x) = \sum_{ n \in \mathbb{Z}}x^{n}.
\end{align*}
Certain elementary identities concerning delta functions are very
convenient for dealing with the arithmetic of vertex algebras and, in
fact, in some cases, are fundamental to the very notion of vertex
algebra.  We state and prove some such identities in this section.
The following identity appeared in \cite{FLM}.  Our proof follows that
given in \cite{R}.
\begin{prop}
\label{delta}
We have the following two elementary identities:
\begin{align}
\label{twotermdelta}
y^{-1}\delta\left(\frac{x-z}{y}\right)-
x^{-1}\delta\left(\frac{y+z}{x}\right)=0
\end{align}
and
\begin{align}
\label{threetermdelta}
z^{-1}\delta\left(\frac{x-y}{z}\right)-
z^{-1}\delta\left(\frac{-y+x}{z}\right)-
x^{-1}\delta\left(\frac{y+z}{x}\right)=0.
\end{align}
\end{prop}
\begin{proof}
First observe that 
\begin{align}
\label{eq:del2}
y^{-1}\delta\left(\frac{x}{y}\right)
=
\sum_{l < 0}x^{l}y^{-l-1}
+
\sum_{l \geq 0}x^{l}y^{-l-1}
=
e^{-y\frac{d}{dx}}x^{-1}+e^{-x\frac{d}{dy}}y^{-1}
=
(x-y)^{-1}+(y-x)^{-1}.
\end{align}
Then by the formal Taylor
theorem
\begin{align*}
y^{-1}\delta\left(\frac{x+z}{y}\right)&=e^{z\frac{d}{dx}}
y^{-1}\delta\left(\frac{x}{y}\right)\\
&=e^{z\frac{d}{dx}}((x-y)^{-1}+(y-x)^{-1})\\
&=((x+z)-y)^{-1}+(y-(x+z))^{-1}.
\end{align*}
Being careful with minus signs, we may respectively expand all the
terms in the left-hand side of (\ref{threetermdelta}) to get
\begin{align*}
((x-y)-z)^{-1}+(z-(x-y))^{-1}\\
-((-y+x)-z)^{-1}-(z-(-y+x))^{-1}\\
-((y+z)-x)^{-1}-(x-(y+z))^{-1}.
\end{align*}
Now we get by Proposition \ref{prop:assocarith} that the first and
sixth terms, the third and fifth terms, and the second and fourth
terms pairwise cancel each other thus giving us
(\ref{threetermdelta}).  The other identity may be proved in a similar
fashion.
\end{proof}
We note the following easily verified identity extending
(\ref{eq:del2}):
\begin{align}
\label{eq:deltder}
\frac{1}{n!}\left(\frac{\partial}{\partial y}\right)^{n}
y^{-1}\delta\left(\frac{x}{y}\right)
=
\frac{(-1)^{n}}{n!}\left(\frac{\partial}{\partial x}\right)^{n}
y^{-1}\delta\left(\frac{x}{y}\right)
=
(x-y)^{-n-1}-(-y+x)^{-n-1}.
\end{align}

In addition to the above identities, delta functions have certain
crucial substitution properties.  To state such a property in the
generality needed, we need the following logarithmic substitution
operator.  It is crucial to note as in \cite{HLZ} that whereas
\begin{align*}
\log(e^{x})=\log(1+(e^{x}-1))=x
\end{align*}
it is {\it not} true that
\begin{align*}
e^{\log x}=x,
\end{align*}
since the left hand side is a series in the variable $\log x$.  But
nonetheless, the two expressions are heuristically equal and we
develop a means of rigorizing this heuristic suggestion next.  For $V$
which does not depend on $\log x$ or $x$ we shall define an operator
\begin{align*}
\phi_{x}:V[[x,x^{-1}]][e^{\log x},e^{-\log x}] \rightarrow V[[x,x^{-1}]].
\end{align*}
The operator $\phi_{x}$ essentially substitutes $x$ for $e^{\log x}$.
For $f(x)= \sum_{n \in \mathbb{Z}}v_{n}(x)e^{n
  \log x} \in V[[x,x^{-1}]][e^{\log x},e^{-\log x}]$ where $v_{n}(x)
\in V[[x,x^{-1}]]$, let
\begin{align*}
\phi_{x}(f(x))= \sum_{n \geq 0}v_{n}(x)x^{n}.
\end{align*}
This is well defined because $v_{n}(x)x^{n}$ is a well defined formal
Laurent series and $\phi_{x}(f(x))$ is a finite sum of such series.
We shall use the notation
\begin{align*}
\phi_{x,y}=\phi_{x}\circ \phi_{y}
\end{align*}
etc., if we wish to indicate the operator substituting for more than
one logarithmic variable in a similar way.

\begin{prop}
The operator $\phi_{x}$ commutes with $\frac{d}{dx}$.
\end{prop}
\begin{proof}
  It is enough to consider acting on elements of the form $x^{n}e^{m
    \log x}$.  We have
\begin{align*}
\phi_{x} \frac{d}{dx}x^{n}e^{m \log x}
=\phi_{x} \left((n+m)x^{n-1}e^{m \log x}\right)
=(m+n)x^{m+n-1}
\end{align*}
and
\begin{align*}
\frac{d}{dx} \phi_{x} x^{n}e^{m \log x}
&=
\frac{d}{dx}x^{n+m}=(m+n)x^{m+n-1}.
\end{align*}
\end{proof}
The following is an example of how we shall typically be applying $\phi_{x}$,
\begin{align*}
\phi_{x}e^{-2 \log (x-z)}
=\phi_{x}e^{-z\frac{\partial}{\partial x}}e^{-2 \log x}
=e^{-z\frac{\partial}{\partial x}} \phi_{x} e^{-2 \log x}
=e^{-z\frac{\partial}{\partial x}}x^{-2}
=(x-z)^{-2}.
\end{align*}

\begin{prop}
\label{prop:deltalogsub}
For 
\begin{align*}
f(x,y,z) \in \text{\rm End}\,V[[x,x^{-1},e^{\log x},e^{-\log
  x},y,y^{-1},e^{\log y},e^{-\log y},z,z^{-1},e^{\log z},e^{-\log
  z}]]
\end{align*}
such that for each fixed $v \in V$
\begin{align*}
  f(x,y,z)v \in \text{\rm End}\,V[[x,x^{-1},y,y^{-1}]][e^{\log
    x},e^{-\log x},e^{\log y},e^{-\log y}]((z))[e^{\log z},e^{-\log
    z}]
\end{align*}
and such that
\begin{align*}
\text{\rm lim}_{x \rightarrow y}f(x,y,z)
\end{align*}
exists (where the ``limit'' is the indicated formal substitution), we
 have
\begin{align*}
\phi_{x,y,z}\delta\left(\frac{y+z}{x}\right)f(x,y,z)
&=
\phi_{x,y,z}\delta\left(\frac{y+z}{x}\right)f(y+z,y,z)\\
&=
\phi_{x,y,z}\delta\left(\frac{y+z}{x}\right)f(x,x-z,z).
\end{align*}
\end{prop}
\begin{proof}
Since we are substituting terms homogeneous of degree 1, for other
terms homogeneous of degree 1, therefore after $\phi_{x,y,z}$ is
applied we need only consider homogeneous $f(x,y,z)$.  The limit
restriction on $f(x,y,z)$ together with the fact that the logarithmic
terms only involve finite sums means that we we need only consider
when
\begin{align*}
f(x,y,z)
=
x^{l_{1}}e^{l_{2}\log x}y^{m_{1}}e^{m_{2}\log y}z^{n_{1}}e^{n_{2}\log z}.
\end{align*}
This follows from an easy calculation.
\end{proof}

\section{Some general log vertex operators}
Let $x,\log x$ and $\xi_{i}$ for $i \geq 0$ be formal commuting
variables and consider the space
$\mathbb{C}[\xi_{1},\xi_{2},\dots][[\xi_{0}]]$.  We define some
operators on this space (cf. (4.2.1) in \cite{FLM}).  Let
\begin{align*}
h^{+}(x)
=
2\log x \frac{\partial}{\partial \xi_{0}}  
-2\sum_{n > 0} \frac{x^{-n}}{n}\frac{\partial}{\partial \xi_{n}}
\end{align*}
and
\begin{align*}
h^{-}(x)=\sum_{m \geq 0} \xi_{m}x^{m}.
\end{align*}
Let 
\begin{align*}
h(x)=h^{+}(x)+h^{-}(x),
\end{align*}
so that
\begin{align*}
h(x):\mathbb{C}[\xi_{1},\xi_{2},\dots][[\xi_{0}]]
\rightarrow
\mathbb{C}[\xi_{1},\xi_{2},\dots][[\xi_{0}]]((x))[\log
x].
\end{align*}

We use colons to denote normally ordered products, where normal
ordering places differential operators to the right side of products.
For instance, for $n \geq 0$
\begin{align*}
:\frac{\partial}{\partial \xi_{n}}\xi_{n}:
=
:\xi_{n}\frac{\partial}{\partial \xi_{n}}:
=
\xi_{n}\frac{\partial}{\partial \xi_{n}},
\end{align*} 
with the notation extended in the obvious way over longer products.
We have
\begin{align*}
:h^{-}(x)h^{+}(x):=:h^{-}(x)h^{+}(x):=h^{-}(x)h^{+}(x),
\end{align*}
so that for $m \in \mathbb{Z}$ and $n \geq 0$
\begin{align}
\label{eq:hpower}
:mh(x)^{n}:=:(mh^{-}(x)+mh^{+}(x))^{n}:
=\sum_{l+k=n}
\frac{(mh^{-}(x))^{l}}{l!}
\frac{(mh^{+}(x))^{k}}{k!}.
\end{align}

Let us consider $:mh(x)^{n}:$ acting on an element of
$\mathbb{C}[\xi_{1},\xi_{2},\dots][[\xi_{0}]]$.  In particular,
focusing on the coefficient of $x$ and separately $\log x$ and
separately $\xi_{0}$ all raised to fixed powers then it is clear that
the answer is zero for sufficiently large $n$.  Therefore dividing by
$n!$ and summing, we get the well defined operator (cf. (4.2.6) and
also (4.2.12) in \cite{FLM})
\begin{align*}
:e^{mh(x)}:.
\end{align*}
Recall (\ref{Xi}), the space
\begin{align*}
\Xi=\mathbb{C}[e^{\xi_{0}},
e^{-\xi_{0}},\xi_{1},\xi_{2},\dots].
\end{align*}
We may consider $:e^{mh(x)}:$ acting on $\Xi$ as in the next
proposition.
\begin{prop}
  For $m \in \mathbb{Z}$ $:e^{mh(x)}::\Xi \rightarrow
  \Xi((x))[e^{\log x},e^{-\log x}]$
\begin{align*}
  :e^{mh(x)}: = e^{mh^{-}(x)}e^{mh^{+}(x)} =
  e^{m\xi_{0}}e^{mh^{-}(x)-m\xi_{0}}e^{mh^{+}(x)-2 \log x
    \frac{\partial}{\partial \xi_{0}}}e^{2 \log x
    \frac{\partial}{\partial \xi_{0}} }.
\end{align*}
\end{prop}             
\begin{proof}
  Divide (\ref{eq:hpower}) by $n!$ and sum over $n \geq 0$.  That we
  may restrict the domain and range follows easily from the expanded
  formula.
\end{proof}

We define a linear and ``normal multiplicative'' map on
$\Xi$.  We shall explain ``normal multiplicative'' in a
moment.  First let
\begin{align}
\label{vac}
\Upsilon(1,x)=1.
\end{align}
Next, for $n \geq 1$, let
\begin{align*}
\Upsilon(\xi_{n},x)=\frac{h^{(n)}(x)}{n!}.
\end{align*}
And extend naturally to let
\begin{align}
\label{Upsilon}
\Upsilon(e^{m\xi_{0}},x)=:e^{mh(x)}:,
\end{align}
for $m \in \mathbb{Z}$. 
\begin{remark} \rm
The reader should compare the operator (\ref{Upsilon}) with the
heuristic operator in Remark 4.2.1 of \cite{FLM}, which provided the
original motivation for developing this, now rigorous, operator.
\end{remark}
Then for $p,q \in \Xi$ we require that
\begin{align}
\label{*0}
\Upsilon(pq,x)=:\Upsilon(p,x)\Upsilon(q,x):,
\end{align}
(which is what we mean by ``normal multiplicative'').  This determines a
unique linear map
\begin{align*}
\Upsilon(\cdot,x)\cdot: 
\Xi \otimes \Xi 
\rightarrow 
\Xi((x))[e^{\log x},e^{-\log x}].
\end{align*}

Extending the notation in the obvious way, we have
\begin{align*}
\Upsilon(h^{-}(z),x)
=\Upsilon
\left(
\sum_{m \geq 0}\xi_{m}z^{m},x 
\right)
=\sum_{m \geq 0}\frac{h^{(m)}(x)}{m!}z^{m}
=e^{z\frac{d}{dx}}h(x)
=h(x+z).
\end{align*}
So
\begin{align}
\label{*1}
\Upsilon(e^{h^{-}(z)},x)=:e^{h(x+z)}:,
\end{align}
and the range of this operator is $\Xi((x))[e^{\log
  x},e^{-\log x}][[z]]$, which is easy to see using Proposition
\ref{prop:logtaylor}.  More generally, the range of
\begin{align*}
:e^{h(x+y)}e^{h(x+z)}:
\end{align*}
is $\Xi((x))[e^{\log x},e^{-\log x}][[y,z]]$ which can be
seen by first considering
\begin{align*}
:e^{h(x_{1}+y)}e^{h(x_{2}+z)}:
\end{align*}
and then setting $x_{1}$ and $x_{2}$ equal to $x$.  We shall use such
reasoning without comment below.

We have
\begin{align*}
h^{+}(x)e^{nh^{-}(y)}
&=
\left(2 \log x \frac{\partial}{\partial \xi_{0}} 
-2\sum_{k > 0} \frac{x^{-k}}{k}\frac{\partial}{\partial \xi_{k}}\right)
e^{n\sum_{m \geq 0} \xi_{m}y^{m}}\\
&=\left(2n \log x - 2\sum_{k > 0}\frac{n}{k}y^{k}x^{-k}\right)
e^{nh^{-}(y)}\\
&=2n\left(\log x + \log \left(1-\frac{y}{x}\right)\right)
e^{nh^{-}(y)}\\
&=2n\log (x-y)
e^{nh^{-}(y)},\\
\end{align*}
so that, letting $k=2mn$,
\begin{align}
\label{*2}
e^{-mh^{+}(x)}e^{nh^{-}(y)}=e^{-k\log(x-y)}e^{nh^{-}(y)}.
\end{align}
We note that we may set $y=0$ in this formula.  We will use this
later.
\section{Jacobi Identity}
For $m,n \in \mathbb{Z}$, recalling (\ref{eq:ecomm})
\begin{align*}
:e^{mh(x)}::e^{nh(z)}:
&=
e^{mh^{-}(x)}
e^{mh^{+}(x)}
e^{nh^{-}(z)}
e^{nh^{+}(z)}
=
:e^{mh(x)+nh(z)}:
e^{[mh^{+}(x),nh^{-}(z)]},
\end{align*}
because (cf. the proof of Proposition 4.3.1 in \cite{FLM})
\begin{align*}
\left[h^{+}(x),h^{-}(z)\right]
&=\left[2\frac{\partial}{\partial \xi_{0}} \log x 
-2\sum_{n > 0} \frac{1}{n}\frac{\partial}{\partial \xi_{n}}x^{-n}
,
\sum_{m \geq 0} \xi_{m}z^{m}
\right]\\
&=\left[
2\frac{\partial}{\partial \xi_{0}},\xi_{0}\right]\log x
+\sum_{n,m >0}
\left[
\frac{-2}{n}\frac{\partial}{\partial \xi_{n}}
,
\xi_{m}
\right]
x^{-n}z^{m}\\
&=2 \log x 
-2 \sum_{n > 0}\frac{1}{n}
\left(
\frac{z}{x}
\right)^{n}\\
&=2 \log x +2 \log \left(1-\frac{z}{x}\right)\\
&=2 \log (x-z)
\end{align*}
is central, and moreover we have
\begin{align}
\label{*3}
:e^{mh(x)}::e^{nh(z)}:
=
:e^{mh(x)+nh(z)}:
e^{2mn \log (x-z)}
.
\end{align}

The reader should compare the calculation here with the work in
Section 8.6 of \cite{FLM}.  Let $x_{i},y_{i},z_{i}$ for $i \geq 1$ and
for convenience we also let $x_{0}=y_{0}=z_{0}=0$.  Then let
\begin{align*}
A=e^{\sum_{i \geq 1}-m_{i}h^{-}(x_{i})}e^{-m_{0}\xi_{0}}
\end{align*}
and
\begin{align*}
B=e^{\sum_{j \geq 1}n_{j}h^{-}(y_{j})}e^{n_{0}\xi_{0}},
\end{align*}
where $m_{i},n_{j} \in \mathbb{Z}$ and only finitely many are nonzero
(cf. (8.6.5) in \cite{FLM}).  Then letting $k_{ij}=2m_{i}n_{j}$,
\begin{align*}
&
\phi_{x,y,z}
\Upsilon\left(A,x\right)
\Upsilon\left(B,y\right)\\
&\quad=
\phi_{x,y,z}
:e^{\sum_{i \geq 0}-m_{i}h(x+x_{i})}:
:e^{\sum_{j \geq 0}n_{j}h(y+y_{j})}:\\
&\quad=
\phi_{x,y,z}
:e^{\sum_{i \geq 0}-m_{i}h(x+x_{i})+\sum_{j \geq 0}n_{j}h(y+y_{j})}:
e^{\sum_{i,j \geq 0}-k_{ij}\log(x+x_{i}-y-y_{j})}\\
&\quad=
\phi_{x,y,z}
:e^{\sum_{i \geq 0}-m_{i}h(x+x_{i})+\sum_{j \geq 0}n_{j}h(y+y_{j})}:
\prod_{i,j \geq 0}(x+x_{i}-y-y_{j})^{-k_{ij}},
\end{align*}
by (\ref{*0}), (\ref{*1}) and (\ref{*3}) and its range is
$\Xi((x))[[x_{1},x_{2},\dots]]((y))[[y_{1},y_{2},\dots]]$.
Therefore we may multiply this expression by
$z^{-1}\delta\left(\frac{x-y}{z}\right)$ and Proposition
\ref{prop:deltalogsub} may be applied to give
\begin{align}
\label{eq:ljac1}
\phi_{x,y,z}
&z^{-1}\delta\left(\frac{x-y}{z}\right) 
\Upsilon\left(A,x\right)
\Upsilon\left(B,y\right) \nonumber\\
&=
\phi_{x,y,z}
z^{-1}\delta\left(\frac{x-y}{z}\right) 
:e^{\sum_{i \geq 0}-m_{i}h(x+x_{i})+\sum_{j \geq 0}n_{j}h(y+y_{j})}:
\prod_{i,j \geq 0}(z+x_{i}-y_{j})^{-k_{ij}}.
\end{align}
Similarly
\begin{align}
\label{eq:ljac2}
\phi_{x,y,z}
&z^{-1}\delta\left(\frac{y-x}{-z} \right) 
\Upsilon\left(B,y\right)
\Upsilon\left(A,x\right) \nonumber\\
&=
\phi_{x,y,z}
z^{-1}\delta\left(\frac{y-x}{-z} \right)
:e^{\sum_{i \geq 0}-m_{i}h(x+x_{i})+\sum_{j \geq 0}n_{j}h(y+y_{j})}:
\prod_{i,j \geq 0}(z+x_{i}-y_{j})^{-k_{ij}},
\end{align}
noting that $(-z-x_{i}+y_{j})^{-k_{ij}}=(z+x_{i}-y_{j})^{-k_{ij}}$,
since $k_{ij}$ is always even.

Further
\begin{align*}
\Upsilon\left(
\Upsilon\left(
A,z\right)
B,y\right)
&=
\Upsilon
\left(
:e^{\sum_{i \geq 0}-m_{i}h(z+x_{i})}:
e^{\sum_{j \geq 1}n_{j}h^{-}(y_{j})+n_{0}\xi_{0}},y\right)\\
&=e^{\sum_{i,j \geq 0}-k_{ij}\log(z+x_{i}-y_{j})}
\Upsilon\left(
e^{\sum_{i \geq 0}-m_{i}h^{-}(z+x_{i})+
\sum_{j \geq 1}n_{j}h^{-}(y_{j})+n_{0}\xi_{0}}
,y\right)\\
&=e^{\sum_{i,j \geq 0}-k_{ij}\log(z+x_{i}-y_{j})}
:e^{\sum_{i \geq 0}-m_{i}h(y+z+x_{i})+\sum_{j \geq 0}n_{j}h(y+y_{j})}:,
\end{align*}
by (\ref{*0}), (\ref{*1}) and (\ref{*2}), and its range is
\begin{align*}
\Xi((y))[e^{\log y},e^{-\log y}][[y_{1},y_{2},\dots]]
((z))[e^{\log z},e^{-\log z}][[x_{1},x_{2},\dots]],
\end{align*}
so that we may multiply by $y^{-1}\delta\left(\frac{x-z}{y} \right)$
and Proposition \ref{prop:deltalogsub} applies, giving
\begin{align}
\label{eq:ljac3}
\phi_{x,y,z}
&x^{-1}\delta\left(\frac{y+z}{x} \right) 
\Upsilon\left(
\Upsilon\left(
A,z\right)
B,y\right) \nonumber\\
&=
\phi_{x,y,z}
x^{-1}\delta\left(\frac{y+z}{x} \right)
\prod_{i,j \geq 0}(z+x_{i}-y_{j})^{-k_{ij}}
:e^{\sum_{i \geq 0}-m_{i}h(x+x_{i})+\sum_{j \geq 0}n_{j}h(y+y_{j})}:.
\end{align}
Now adding the left hand sides of (\ref{eq:ljac1}), (\ref{eq:ljac2})
and (\ref{eq:ljac3}) together, factoring and using
Proposition~\ref{delta}, we get:
\begin{align}
\label{eq:jacgen}
\phi_{x,y,z}
z^{-1}&\delta\left(\frac{x-y}{z}\right) 
\Upsilon\left(A,x\right)
\Upsilon\left(B,y\right) 
-
\phi_{x,y,z}
z^{-1}\delta\left(\frac{y-x}{-z} \right) 
\Upsilon\left(B,y\right)
\Upsilon\left(A,x\right) \nonumber\\
&\quad=
\phi_{x,y,z}
x^{-1}\delta\left(\frac{y+z}{x} \right) 
\Upsilon\left(
\Upsilon\left(
A,z\right)
B,y\right).
\end{align}
\section{A vertex algebra}
The operator $\Upsilon(\cdot,x)$ is almost a vertex operator, but we
need to specialize $e^{\log x}$ to $x$.  So we let
\begin{align}
\label{vop}
Y(\cdot, x)\cdot =\phi_{x}\circ \Upsilon(\cdot, x)\cdot 
\end{align}
so that 
\begin{align*}
Y(\cdot,x)\cdot:\Xi \otimes \Xi \rightarrow \Xi((x)) 
\end{align*}
is a linear map.

We are now ready to recall the definition of a vertex algebra.
Individual mathematical vertex operators were introduced in \cite{LW}.
The notion of vertex algebra was first mathematically defined by
Borcherds in \cite{B}.  An equivalent set of axioms, based primarily
on the Jacobi identity, which we shall use, appeared first in
\cite{FLM} as part of the notion of vertex operator algebra (cf.
\cite{LL}, in which the equivalence of the two sets of axioms is
proved).
\begin{defi} \rm
A \it{vertex algebra} \rm is a vector space equipped,
first, with a linear map (the \it{vertex operator map}) $V
\otimes V \rightarrow V[[x,x^{-1}]]$, \rm or equivalently, a linear
map
\begin{align*}
Y(\,\cdot\,,x): \quad &V \, \rightarrow \, (\text{\rm End}V)[[x,x^{-1}]]\\
&v \, \mapsto \, Y(v,x)=\sum_{n \in \mathbb{Z}}v_{n}x^{-n-1}.
\end{align*}
We call $Y(v,x)$ the \it{vertex operator associated with} $v$.  \rm We
assume that
\begin{align*}
Y(u,x)v \in V((x))
\end{align*}
for all $u,v \in V$.  There is also a distinguished 
element $\textbf{1}$ satisfying the following
\it{vacuum property} \rm:
\begin{align*}
Y(\textbf{1},x)=1
\end{align*}
and \it{creation property} \rm:
\begin{align*}
Y(u,x)\textbf{1} \in V[[x]] \text{ and}\\
Y(u,0)\textbf{1}=u  \text{ for all }u \in V.
\end{align*}
Finally, we require that the
\it{Jacobi identity} \rm is satisfied:
\begin{align*}
z^{-1}\delta\left(\frac{x-y}{z}\right) 
Y(u,x)Y(v,y)&-
z^{-1}\delta\left(\frac{y-x}{-z} \right) 
Y(v,y)Y(u,x)\\
&=
x^{-1}\delta\left(\frac{y+z}{x} \right) 
Y(Y(u,z)v,y).
\end{align*}
\end{defi}
\begin{theorem}
\label{ert}
  The space $\Xi$ equipped with $Y(\cdot,x)\cdot$ as defined
  by (\ref{vop}), is a vertex algebra, with vacuum vector $1$.
\end{theorem}
\begin{proof}
  The minor conditions, such as truncation have already been dealt
  with.  The vacuum property follows by (\ref{vac}).  To see the
  creation property, note that
\begin{align*}
Y(A,x)1&=\phi_{x}:e^{\sum_{i \geq 0}-m_{i}h(x+x_{i})}:1
=e^{\sum_{i \geq 0}-m_{i}h^{-}(x+x_{i})},
\end{align*}
which is a power series in $x$ and we may substitute $x=0$ which is
easily seen to give
\begin{align*}
Y(A,0)1=A.
\end{align*}
Noting (\ref{eq:jacgen}), we will be done if we can show that the
coefficients in $A$ and $B$ span $\Xi$ as we let their parameters
vary.

The following is essentially the same observation as in Remark 8.3.9
in \cite{FLM}.  Let
\begin{align*}
e^{\sum_{n \geq 1}\xi_{n}y^{n}}=\sum_{k \geq 0}p_{k}y^{k}
\quad \text{so that} \quad
\sum_{n \geq 1}\xi_{n}y^{n}
=\log\left(1+\sum_{k \geq 1}p_{k}y^{k}\right),
\end{align*}
which shows that any polynomial in the $\xi_{i}$ $i \geq 1$ is a
polynomial in the $p_{i}$ for $i \geq 1$.  This gives the result.
\end{proof}

\noindent {\small \sc Department of Mathematics, Rutgers University,
Piscataway, NJ 08854} 
\\ {\em E--mail
address}: thomasro@math.rutgers.edu

\end{document}